\documentclass{amsart}
\usepackage{amsfonts}
\usepackage{amsmath}
\usepackage{amssymb}
\usepackage{graphicx}
\usepackage{hyperref}
\providecommand{\U}[1]{\protect\rule{.1in}{.1in}}
\newtheorem{theorem}{Theorem}
\theoremstyle{plain}

\newtheorem{corollary}{Corollary}

\newtheorem{definition}{Definition}
\newtheorem{example}{Example}

\newtheorem{lemma}{Lemma}

\newtheorem{proposition}{Proposition}
\newtheorem{remark}{Remark}

\newtheorem{question}{Question}
\numberwithin{equation}{section}

\begin{document}
\title[{\normalsize On annihilator multiplication modules}]{{\normalsize On annihilator multiplication modules}}
\author{Suat Ko\c{c}}
\address{Marmara University, Department of Mathematics, Istanbul, T\"{u}rkiye.}
\email{suat.koc@marmara.edu.tr}
\begin{abstract}
An $A$-module $E$ is said to be an  \textit{annihilator multiplication module} if for each $e\in E$, there exists a finitely generated ideal $I$ of $A$ such that $ann(e)=ann(IE)$. This class of modules is quite large, as it contains multiplication modules, von Neumann regular modules, finitely generated Baer modules, torsion-free modules, and simple modules. This article provides a comprehensive investigation into the algebraic properties of annihilator multiplication modules, and establishes new characterizations for several important classes of rings/modules, including multiplication modules, torsion-free modules, simple modules, uniserial modules, injective modules and Noetherian von Neumann regular rings. Furthermore, we present a construction method using trivial extensions to produce annihilator multiplication rings that are not multiplication rings. In addition, we prove that, for such modules, the equality $Ass_{A}(E)=Ass(A)$ holds, thereby providing a precise connection between module-theoretic and ring-theoretic prime structures. Finally, we provide various examples to demonstrate the above equality may fail if the condition of being annihilator multiplication module is omitted.

\end{abstract}
\maketitle

\section{Introduction}

The fundamental interplay between the structural properties of a ring and those of its modules is a cornerstone of commutative algebra. A central theme in this area is identifying classes of modules that mirror the ideal structure of their underlying rings. The concept of \textbf{multiplication modules}, first introduced by Barnard in \cite{Barnard} and further developed by El-Bast and Smith in \cite{Elbast}, has long served as a crucial bridge in this context. An $A$-module $E$ is said to be a \textit{multiplication module} if every submodule $V$ of $E$ is of the form $IE$ for some ideal $I$ of $A$. Particularly, an ideal $I$ of a ring $A$ is called a \textit{multiplication ideal} if it is a multiplication $A$-module, and also $A$ is said to be a \textit{multiplication ring} if each of its ideals is a multiplication ideal \cite{Larsen}. In such modules, the algebraic properties of $V$ are frequently compatible with those of the ideal $(V:_{A}E)=ann_{A}(E/V)$, allowing for a rich translation of ring-theoretic results into the language of module theory. For this reason, multiplication modules have become a standard tool for researchers seeking to extend classical ring results to a broader module-theoretic setting (See, for example, \cite{Ali}, \cite{Ameri}, \cite{ArAnTeKo}, \cite{Azizi}, \cite{CeTeKo}, \cite{MacMo}, \cite{Ones}, \cite{Tuganbaev}). Another factor showing the importance of multiplication modules/rings is that these class of rings/modules contains simple modules, cyclic modules, von Neumann regular rings/modules, and also they are used in the characterization of various significiant rings/modules. For instance, arithmetical rings (a ring $A$ is said to be an \textit{arithmetical ring} if $A_{\frak{m}}$ is a valuation ring for each maximal ideal $\frak{m}$ of $A$ \cite{Larsen}) are characterized by means of multiplication ideals as follows: a ring $A$ is an arithmetical ring if and only if its each finitely generated ideal is a multiplication ideal (See, \cite[Page 151, Exercise 18]{Larsen}). Furthermore, a fundamental structural property of multiplication modules is their significance in spectral theory. Specifically, multiplication modules offer a natural framework for extending the classical Zariski topology, which is traditionally defined on the prime spectrum of a ring, to the spectrum of a module. More precisely, while the collection of varieties defined on the spectrum of a general module typically fails to satisfy the axioms for closed sets, the multiplication module property ensures that these varieties form a well-defined topology, which is known as quasi-Zariski topology \cite{MacMo}.

Nevertheless, the concept of a multiplication module is limited by its restrictive defining conditions. As a result, it excludes several important structures in module theory, including certain torsion-free modules, free modules of $rank\geq 2$, vector spaces of $dimension\geq 2$, etc. Also, the multiplication module is sometimes insufficient for establishing essential correspondences between rings and modules, especially for Baer structures, since the classical multiplication module property does not adequately connect Baer modules with their underlying Baer rings. In 2022, Jayaram, Tekir, and Koç introduced annihilator multiplication modules to fully reflect the relationship between Baer modules and Baer rings. By overcoming the restrictions of multiplication modules, annihilator multiplication modules provide an extensive framework including multiplication modules, torsion-free modules, free modules (and thus vector spaces), and finitely generated Baer modules. The fundamental distinction of this approach lies in establishing an identity based on the annihilators of factor modules, rather than defining submodules through a direct equality $V=IE$. In this paper, we aim to deepen the theory by investigating this class as a significant entity in its own right. Beyond examining annihilator multiplication modules under various algebraic structures, we demonstrate how this module class can be utilized to characterize several important modules and rings. Furthermore, we prove that annihilator multiplication modules establish a connection between the associated primes of the module and those of the underlying ring. We now begin by providing some notions and notations regarding rings and modules that will be used in the sequel.

Throughout the paper, we focus only on commutative rings with a nonzero identity and nonzero unital modules. Let $A$ always represents such a ring and $E$ represents such an $A$-module. A ring $A$ (not necessarily commutative) is defined as \textit{von Neumann regular} if, for each $a$ in $A$, there exists $x$ in $A$ such that $a = a x a$ \cite{von}. Von Neumann originally introduced these rings due to their connection with continuous geometry \cite{von}. Their algebraic significance has become increasingly recognized in recent years. In 2018, Jayaram and Tekir extended the concept of von Neumann regular rings to modules by introducing $E$-von Neumann regular and weak idempotent elements. Specifically, an element $a$ in $A$ is said to be an \textit{$E$-von Neumann regular} (or \textit{weak idempotent}) if $aE = a^2E$ (or $a - a^2 \in ann_{A}(E)$). An $A$-module $E$ is called a \textit{von Neumann regular} (for short, \textit{vn-regular}) \textit{module} if, for each $e$ in $E$, $Ae = aE$ for some $E$-von Neumann regular element $a\in A$ \cite{JaTe}. It is easy to see that every vn-regular module (vn-regular ring) is also a multiplication module (a multiplication ring). Furthermore, Jayaram and Tekir demonstrated that a finitely generated $A$-module $E$ is vn-regular if and only if, for each $e$ in $E$, there exists a weak idempotent $a$ in $A$ such that $Ae = aE$ \cite[Lemma 5]{JaTe}. In functional analysis, Baer $*$-rings serve as the algebraic counterparts of von Neumann algebras. Every von Neumann algebra is a Baer $*$-ring because the right annihilator of any subset is generated by a projection. This relationship underscores the structural importance of projections in operator algebras and positions Baer $*$-rings as an algebraic framework for the study of annihilators and projection lattices. Kaplansky defined a \textit{Baer ring} as a ring $A$ in which the annihilator $ann(S) = \{a \in A : aS = 0\}$ of any subset $S$ of $A$ is generated by an idempotent $b$ in $A$ \cite{Kaplansky}. Subsequently, Kist refined this definition by considering $S$ as a singleton, rather than an arbitrary subset. According to Kist \cite{Kist}, a ring $A$ is a \textit{Baer ring} (also referred to as p.q. Baer or P.P. ring) if, for every $a$ in $A$, there exists an idempotent $b$ in $A$ such that $ann(a) = bA$. The present paper adopts Kist's definition of a Baer ring. Let $V$ be a submodule of $E$ and $K$ a nonempty subset of $E$. \textit{The residual} of $V$ by $K$ is denoted by $(V:_{A}K)=\left\{a\in A:aK\subseteq V\right\}$. When $V=(0)$ and $K$ is a submodule of $E$, we prefer $ann(K)$ to denote $((0):K)$. In particular, for each $e\in E$, we use $ann(e)$ instead of $ann(Ae)$. An $A$-module $E$ is said to be a \textit{torsion-free} if $ann(e)=0$ for every $0\neq e\in E$ \cite{Sharp}. $E$ is said to be a \textit{faithful module} if $ann(E)=0$. It is clear that all torsion-free modules are faithful, and the converse is not true in general. For instance, $%
%TCIMACRO{\U{2124} }%
%BeginExpansion
\mathbb{Z}
%EndExpansion
$-module $%
%TCIMACRO{\U{2124} }%
%BeginExpansion
\mathbb{Z}
%EndExpansion
\oplus%
%TCIMACRO{\U{2124} }%
%BeginExpansion
\mathbb{Z}
%EndExpansion
_{n}$, where $n\geq2$ is a faithful module which is not torsion-free. Recall from \cite{JaTeKo} that an $A$-module $E$ is said to be a \textit{Baer module} if for each $e\in E$ there exists a weak idempotent $a\in A$ such that $ann(e)E=aE$. Also, they defined an $A$-module $E$ to be an \textit{annihilator multiplication} if for each $e\in E$, there exists a finitely generated ideal $I$ of $A$ such that $ann(e)=ann(IE)$. The authors in \cite[Proposition 2.8]{JaTeKo} showed that  every finitely generated von Neumann regular module is a Baer module, and also they proved that every finitely generated Baer module is also an annihilator multiplication module (See, \cite[Lemma 2.1]{JaTeKo}). Furthermore, \cite{JaTeKo} established that annihilator multiplication modules provide a link between the theory of Baer modules and Baer rings. Specifically, a finitely generated $A$-module $E$ is a Baer module if and only if $E$ is an annihilator multiplication module and $A/ann(E)$ is a Baer ring \cite[Theorem 2.14]{JaTeKo}. 

In this article, we aim to investigate the comprehensive algebraic properties of annihilator multiplication modules. In Section 2, we analyze the relationships between annihilator multiplication modules and other classical modules, including multiplication modules, von Neumann regular modules, finitely generated Baer modules, torsion-free modules, and simple modules (see Example \ref{prop1} and Example \ref{ex1}). We also examine the stability of annihilator multiplication modules under direct products, in direct sums, under localizations, under homomorphisms, in factor modules, and polynomial modules (see Proposition \ref{direct}, Proposition \ref{pdirectsum}, Proposition \ref{quotient}, Proposition \ref{phom}, Corollary \ref{csub}, Proposition \ref{ppol}). For an ideal $I$ of $A$, we say that $I$ is an \textit{annihilator multiplication ideal} if it is an annihilator multiplication $A$-module. Also, we define a ring $A$ as an \textit{annihilator multiplication ring} if its each ideal is an annihilator multiplication ideal. We investigate certain annihilator multiplication ideals and multiplication ideals in trivial ring extension (See, Theorem \ref{tri1}, Theorem \ref{tri2} and Theorem \ref{tri3}). In terms of these results, we characterize simple modules, and also we provide a construction method using trivial ring extensions to produce annihilator multiplication rings that are not multiplication rings (See, Theorem \ref{tri3} and Example \ref{construction}). In Section 3, we study 1-absorbing prime ideals and classical 1-absorbing prime submodules in annihilator multiplication modules (see Proposition \ref{1-abs} and Proposition \ref{pclass}). Also, we characterize some significiant classes of rings and modules, such as torsion-free modules, multiplication modules, injective modules, uniserial modules and Noetherian von Neumann regular rings, using annihilator multiplication modules (See, Theorem \ref{ttorsion}, Proposition \ref{pmult}, Proposition \ref{pinj}, Corollary \ref{uniserial} and Theorem \ref{tvon}). Finally, we prove that $Ass_{A}(E)=Ass(A)$ in annihilator multiplication modules, where $Ass_{A}(E)$ and $Ass(A)$ denote the sets of associated primes of $E$ and $A$, respectively (see Theorem \ref{tassoc}). Also, we provide various examples to demonstrate the above equality may fail if the condition of being annihilator multiplication module is omitted (See, Example \ref{exassociated}).

\section{Basic properties of annihilator multiplication modules}

\begin{definition}
An $A$-module $E$ is said to be an annihilator
multiplication module if for each $e\in E$, there exists a finitely generated ideal $I$ of $A$ such that $ann(e)=ann(IE)$ \cite{JaTeKo}. In particular, an ideal $I$ of $A$ is an annihilator multiplication ideal if it is an annihilator multiplication $A$-module. Also, a ring $A$ is said to be an annihilator multiplication ring if its each ideal is an annihilator multiplication ideal.
\end{definition}

Now, we start by giving some examples of annihilator multiplication modules.

\begin{example}
\label{prop1}(i) Every multiplication module is an annihilator multiplication module. In particular, every vn-regular module is an annihilator multiplication module.

(ii) Every finitely generated Baer module is an annihilator multiplication module.

(iii) Every torsion-free module is an annihilator multiplication module.

(iv) Every simple module is an annihilator multiplication module.
\end{example}

\begin{proof}
$(i):\ $Let $E$ be a multiplication module and choose $e\in E$. Then we have $Ae=(Ae:E)E$. Thus we can write $e=a_{1}e_{1}+a_{2}e_{2}+\cdots+a_{n}e_{n}$ for some $a_{1},a_{2},\ldots,a_{n}\in(Ae:E)$. Now, put $I=%
%TCIMACRO{\tsum \limits_{i=1}^{n}}%
%BeginExpansion
{\textstyle\sum\limits_{i=1}^{n}}
%EndExpansion
Aa_{i}$. Let $v=be\in Ae$. Then we have $v=be=a_{1}(be_{1})+a_{2}(be_{2})+\cdots+a_{n}(be_{n})\in IE$ which implies that $Ae\subseteq
IE\subseteq(Ae:E)E\subseteq Ae$. Thus, we have $Ae=IE$ for some finitely generated ideal $I$ of $A$, and thus we get $ann(e)=ann(IE)$, that is, $E$ is an annihilator multiplication module. The rest follows from the fact that every von Neumann regular module is a multiplication module.

$(ii):\ $Follows from \cite[Lemma 2.11]{JaTeKo}.

$(iii):\ $Follows from \cite[Example 2.10\ (i)]{JaTeKo}.

$(iv):\ $Follows from the fact that every simple module is a multiplication
module and by $(i)$.
\end{proof}

\begin{example}\label{ex1}
Consider $%
%TCIMACRO{\U{2124} }%
%BeginExpansion
\mathbb{Z}
%EndExpansion
$-module $E=%
%TCIMACRO{\U{2124} }%
%BeginExpansion
\mathbb{Z}
%EndExpansion
_{2}\oplus%
%TCIMACRO{\U{2124} }%
%BeginExpansion
\mathbb{Z}
%EndExpansion
_{4}$. Then note that $E$ is not a multiplication module, and also not a
reduced module since $2^{2}(\overline{0},\overline{1})=(\overline{0}%
,\overline{0})$ and $2(\overline{0},\overline{1})\neq(\overline{0},\overline{0})$. As $E$ is finitely generated, by \cite[Proposition 2.7 and
Proposition 2.8]{JaTeKo}, $E$ is neither a Baer module nor a von Neumann
regular module. On the other hand, $E$ is not a torsion-free module and not a
simple module. Let $e=(\overline{x},\overline{y})\in E$. Then it is clear
that $ann(e)=4%
%TCIMACRO{\U{2124} }%
%BeginExpansion
\mathbb{Z}
%EndExpansion
$, $2%
%TCIMACRO{\U{2124} }%
%BeginExpansion
\mathbb{Z}
%EndExpansion
$ or $%
%TCIMACRO{\U{2124} }%
%BeginExpansion
\mathbb{Z}
%EndExpansion
$. Now we have 3 cases. \textbf{Case 1: } If $ann(e)=4%
%TCIMACRO{\U{2124} }%
%BeginExpansion
\mathbb{Z}
%EndExpansion
$, then put $I=%
%TCIMACRO{\U{2124} }%
%BeginExpansion
\mathbb{Z}
%EndExpansion
$ and note that $ann(e)=ann(IE)$. \textbf{Case\ 2:} If $ann(e)=2%
%TCIMACRO{\U{2124} }%
%BeginExpansion
\mathbb{Z}
%EndExpansion
$, then put $I=2%
%TCIMACRO{\U{2124} }%
%BeginExpansion
\mathbb{Z}
%EndExpansion
$ and note that $ann(e)=ann(IE)$ \textbf{Case 3: }If $ann(e)=%
%TCIMACRO{\U{2124} }%
%BeginExpansion
\mathbb{Z}
%EndExpansion
$, then put $I=0%
%TCIMACRO{\U{2124} }%
%BeginExpansion
\mathbb{Z}
%EndExpansion
$ and note that $ann(e)=ann(IE)$. By above all cases, $E$ is an
annihilator multiplication module.
\end{example}

Let $E$ be a multiplication module and $V$ a finitely generated submodule of $E$. According to \cite[Lemma 3.5]{Cho}, there exists a finitely generated ideal $I$ of $A$ such that $V=IE$. The following result provides an analogous statement for annihilator multiplication modules.

\begin{proposition}
\label{prop2}Let $E$ be an annihilator multiplication $A$-module. Then,

(i) If $V$ is a finitely generated submodule of $E$, there exists a finitely
generated ideal $I$ of $A$ such that $ann(V)=ann(IE)$.

(ii) If $V$ is a submodule of $E$ (not necessarily a finitely generated),
then there exists an ideal $I$ of $A$ such that $ann(V)=ann(IE)$.
\end{proposition}

\begin{proof}
$(i):\ $Suppose that $V$ is finitely generated. Then we can write
$V=Av_{1}+Av_{2}+\cdots+Av_{k}$ for some $v_{1},v_{2},\ldots,v_{k}\in V$. This implies that $ann(V)=%
%TCIMACRO{\tbigcap \limits_{i=1}^{k}}%
%BeginExpansion
{\textstyle\bigcap\limits_{i=1}^{k}}
%EndExpansion
ann(v_{i})$. Since $E$ is an annihilator multiplication module, for each
$i=1,2,\ldots,k$, there exist finitely generated ideals $I_{i}$ of $A$ such
that $ann(v_{i})=ann(I_{i}E)$. This gives
\begin{align*}
ann(V)  &  =%
%TCIMACRO{\tbigcap \limits_{i=1}^{k}}%
%BeginExpansion
{\textstyle\bigcap\limits_{i=1}^{k}}
%EndExpansion
ann(v_{i})=%
%TCIMACRO{\tbigcap \limits_{i=1}^{k}}%
%BeginExpansion
{\textstyle\bigcap\limits_{i=1}^{k}}
%EndExpansion
ann(I_{i}E)\\
&  =ann(\left(
%TCIMACRO{\tsum \limits_{i=1}^{k}}%
%BeginExpansion
{\textstyle\sum\limits_{i=1}^{k}}
%EndExpansion
I_{i}\right)  E).
\end{align*}
Since all $I_{i}$'s are finitely generated, so is $%
%TCIMACRO{\tsum \limits_{i=1}^{k}}%
%BeginExpansion
{\textstyle\sum\limits_{i=1}^{k}}
%EndExpansion
I_{i}$, which completes the proof.

$(ii):\ $A similar argument in the proof of $(i)$ shows that
$ann(V)=ann(IE)$ for some ideal $I$ of $A$.
\end{proof}

Let $E_{i}$ be an $A_{i}$-module for each $i=1,2,\ldots,n$, where $n\geq 1$. Assume that $E=E_{1}\oplus E_{2}\oplus\cdots\oplus E_{n}$ and $A=A_{1}\oplus A_{2}\oplus\cdots\oplus A_{n}$. Then $E$ is an $A$-module with component-wise addition and scalar multiplication.

\begin{proposition}\label{direct}
Let $E_{i}$ be an $A_{i}$-module for each $i=1,2,\ldots,n$, where $n\geq
1$. Suppose that $E=E_{1}\oplus E_{2}\oplus\cdots\oplus E_{n}$ and
$A=A_{1}\oplus A_{2}\oplus\cdots\oplus A_{n}$. Then $E$ is an annihilator
multiplication $A$-module if and only if $E_{i}$ is an annihilator
multiplication $A_{i}$-module for each $i=1,2,\ldots,n$.
\end{proposition}

\begin{proof}
$\left(  \Rightarrow\right)  :\ $Let $E$ be an annihilator multiplication
$A$-module. Now, we will show that $E_{1}$ is an annihilator multiplication
$A_{1}$-module. Choose $e_{1}\in E_{1}$ and put $e=(e_{1},0,0,\ldots,0)\in
E$. Since $E$ is an annihilator multiplication module, there exists a
finitely generated ideal $I=I_{1}\oplus I_{2}\oplus\cdots\oplus I_{n}$ of
$A$ such that
\begin{align*}
ann(e)  &  =ann(e_{1})\oplus A_{2}\oplus\cdots\oplus A_{n}=ann(IE)\\
&  =ann(I_{1}E_{1})\oplus ann(I_{2}E_{2})\oplus\cdots\oplus ann(I_{n}E_{n}).
\end{align*}
Then note that $I_{1}$ is a finitely generated ideal of $A_{1}$ and
$ann(e_{1})=ann(I_{1}E_{1})$. Thus, $E_{1}$ is an annihilator multiplication
$A_{1}$-module. Likewise, $E_{i}$ is an annihilator multiplication $A_{i}%
$-module for each $i=2,3,\ldots,n$. 

$\left(  \Leftarrow\right)  :\ $Suppose that $E_{i}$ is an annihilator multiplication $A_{i}$-module for each $i=1,2,\ldots,n$. Let $e=(e_{1}%
,e_{2},\ldots,e_{n})\in E$. Since $E_{i}$ is an annihilator multiplication $A_{i}$-module, there exists a finitely generated ideal $I_{i}$ of $A_{i}$
such that $ann(e_{i})=ann(I_{i}E_{i})$. Then we conclude that
\begin{align*}
ann(e)  &  =ann(e_{1})\oplus ann(e_{2})\oplus\cdots\oplus ann(e_{n})\\
&  =ann(I_{1}E_{1})\oplus ann(I_{2}E_{2})\oplus\cdots\oplus ann(I_{n}E_{n})\\
&  =ann(IE),
\end{align*}
where $I=I_{1}\oplus I_{2}\oplus\cdots\oplus I_{n}$ is an ideal of
$A$. Since $I_{i}$ is a finitely generated ideal of $A_{i}$ for each
$i=1,2,\ldots,n$, it follows that $I$ is a finitely generated ideal of $A$. Hence, $E$ is an annihilator multiplication $A$-module.
\end{proof}

Let $E$ be an $A$-module and $T\subseteq A$ a multiplicatively closed set.
Then we denote the quotient module at $T$ by $T^{-1}E=\left\{  \frac{e}%
{t}:e\in E,t\in T\right\}$ over the quotient ring $T^{-1}A$ \cite{Sharp}.

\begin{proposition}\label{quotient}
Let $E$ be a finitely generated $A$-module and $T\subseteq A$ a multiplicatively closed set. If $E$ is an annihilator multiplication $A$-module, then $T^{-1}E$ is an annihilator multiplication $T^{-1}A$-module.
\end{proposition}

\begin{proof}
Let $\frac{e}{t}\in T^{-1}E$. Since $E$ is an annihilator multiplication $A$-module, there exists a finitely generated ideal $I$ of $A$ such that $ann(e)=ann(IE)$. This implies that
\[
ann_{T^{-1}A}(\frac{e}{t})=T^{-1}\left[  ann(e)\right]  =T^{-1}[ann(IE)].
\]
Since $I$ and $E$ are finitely generated, so is $IE$. This gives
\[
T^{-1}[ann(IE)]=ann_{T^{-1}A}\left[  \left(  T^{-1}I\right)  \left(
T^{-1}E\right)  \right],
\] and so we have $ann_{T^{-1}A}(\frac{e}%
{t})=ann_{T^{-1}A}\left[  \left(  T^{-1}I\right)  \left(  T^{-1}E\right)
\right]$. Since $I$ is a finitely generated ideal of $A$, so is $T^{-1}I$. Thus, $T^{-1}E$ is an annihilator multiplication $T^{-1}A$-module.
\end{proof}

In general, the converse of the previous proposition does not hold. The following example illustrates this point.

\begin{example}
\label{ex2}Let $p$ be a prime number and consider $%
%TCIMACRO{\U{2124} }%
%BeginExpansion
\mathbb{Z}
%EndExpansion
$-module $E=%
%TCIMACRO{\U{2124} }%
%BeginExpansion
\mathbb{Z}
%EndExpansion
_{p}\oplus%
%TCIMACRO{\U{211a} }%
%BeginExpansion
\mathbb{Z}
%EndExpansion
.\ $First note that $E$ is a finitely generated module. Choose
$e=(\overline{1},0)$ and note that $ann(e)=p%
%TCIMACRO{\U{2124} }%
%BeginExpansion
\mathbb{Z}
%EndExpansion
$. Since $%
%TCIMACRO{\U{2124} }%
%BeginExpansion
\mathbb{Z}
%EndExpansion
$ is a PID, every ideal $I\ $has the form $I=n%
%TCIMACRO{\U{2124} }%
%BeginExpansion
\mathbb{Z}
%EndExpansion
$ for some $n\in%
%TCIMACRO{\U{2124} }%
%BeginExpansion
\mathbb{Z}
%EndExpansion
.\ $Then we have two cases. \textbf{Case 1: }If $n=0,$ then $ann(IE)=%
%TCIMACRO{\U{2124} }%
%BeginExpansion
\mathbb{Z}
%EndExpansion
\neq p%
%TCIMACRO{\U{2124} }%
%BeginExpansion
\mathbb{Z}
%EndExpansion
.\ $\textbf{Case 2: }If $n\neq0,\ $then $ann(IE)=(0)\neq p%
%TCIMACRO{\U{2124} }%
%BeginExpansion
\mathbb{Z}
%EndExpansion
$. This shows that $E$ is not an annihilator multiplication module. Let
$T=reg(%
%TCIMACRO{\U{2124} }%
%BeginExpansion
\mathbb{Z}
%EndExpansion
)=%
%TCIMACRO{\U{2124} }%
%BeginExpansion
\mathbb{Z}
%EndExpansion
-\left\{  0\right\}  .\ $Then note that $T^{-1}E$ is a vector space over the
field $T^{-1}%
%TCIMACRO{\U{2124} }%
%BeginExpansion
\mathbb{Z}
%EndExpansion
=%
%TCIMACRO{\U{211a} }%
%BeginExpansion
\mathbb{Q}
%EndExpansion
.\ $Since every vector space is torsion-free, by Example \ref{prop1},
$T^{-1}E$ is an annihilator multiplication module.
\end{example}

\begin{corollary}
Let $E$ be a finitely generated annihilator multiplication module. Then for
every prime ideal $P$ of $A,\ E_{P}$\ is an annihilator multiplication
$A_{P}$-module.
\end{corollary}

\begin{proposition}
\label{phom}Let $\phi:E\rightarrow E^{\prime}$ be an $A$-homomorphism such that
$ann(E)=ann(E^{\prime})$. 

(i) If $\phi$ is one to one and $E^{\prime}$ is an annihilator multiplication module, then $E$ is an annihilator multiplication module.

(ii) If $\phi$ is surjective, $Ker(\phi)$ is a prime submodule of $E$ and $E$ is an annihilator multiplication module, then $E^{\prime}$ is an annihilator multiplication module.
\end{proposition}

\begin{proof}
$(i):\ $Let $e\in E$. Since $E^{\prime}$ is an annihilator multiplication module, there exists a finitely generated ideal $I$ of $A$ such that $ann(\phi(e))=ann(IE^{\prime})$. Since $\phi$ is one to one and $ann(E^{\prime
})=ann(E)$, one can easily see that $ann(e)=ann(\phi(e))$ and
$ann(IE)=ann(IE^{\prime})$ which implies that $ann(e)=ann(IE)$. Thus, $E$ is an annihilator multiplication module.

$(ii):\ $Let $e^{\prime}\in E^{\prime}.\ $If $e^{\prime}=0$, then we have
$ann(e^{\prime})=A=ann((0)E')$, we are done. So assume that $e^{\prime}$ is
nonzero. Since $\phi$ is surjective, we can write $e^{\prime}=\phi(e)\ $for some
$e\in E.\ $As $E\ $is an annihilator multiplication module, there exists a
finitely generated ideal $I\ $of $A\ $such that $ann(e)=ann(IE).\ $Since
$ann(E)=ann(E^{\prime}),\ $we conclude that $ann(IE^{\prime}%
)=ann(IE)=ann(e)\subseteq ann(\phi(e)).\ $Now, let $a\in ann(\phi(e)).\ $Then we
have $a\phi(e)=\phi(ae)=0$ which implies that $ae\in Ker(\phi).\ $Since $Ker(\phi)\ $is a
prime submodule, we have either $a\in(Ker(\phi):E)$ or $e\in Ker(\phi).\ $%
\textbf{Case 1: }Let $a\in(Ker(\phi):E).\ $Then we have $aE\subseteq
Ker(\phi),\ $that is, $\phi(aE)=a\phi(E)=aE^{\prime}=0$ which implies that $a\in
ann(E^{\prime})\subseteq ann(IE^{\prime}).\ $\textbf{Case 2: }Let $e\in
Ker(\phi),\ $that is, $\phi(e)=e^{\prime}=0$ which is a contradiction. Thus, we
conclude that $ann(\phi(e))\subseteq ann(IE^{\prime}),\ $that is, $ann(e^{\prime
})=ann(IE^{\prime})$ which completes the proof.
\end{proof}

In the previous proposition $(ii)$, the condition "$Ker(\phi)\ $is a prime
submodule of $E$" is necessary. See the following example.

\begin{example}
Consider the $%
%TCIMACRO{\U{2124} }%
%BeginExpansion
\mathbb{Z}
%EndExpansion
$-homomorphism $\pi:%
%TCIMACRO{\U{2124} }%
%BeginExpansion
\mathbb{Z}
%EndExpansion
\oplus%
%TCIMACRO{\U{211a} }%
%BeginExpansion
\mathbb{Z}
%EndExpansion
\rightarrow%
%TCIMACRO{\U{2124} }%
%BeginExpansion
\mathbb{Z}
%EndExpansion
_{p}\oplus%
%TCIMACRO{\U{211a} }%
%BeginExpansion
\mathbb{Z}
%EndExpansion
$ defined by $\pi(m,n)=(\overline{m},n),\ $where $p\ $is a
prime number.\ It is clear that $ann(%
%TCIMACRO{\U{2124} }%
%BeginExpansion
\mathbb{Z}
%EndExpansion
\oplus%
%TCIMACRO{\U{211a} }%
%BeginExpansion
\mathbb{Z}
%EndExpansion
)=(0)=ann(%
%TCIMACRO{\U{2124} }%
%BeginExpansion
\mathbb{Z}
%EndExpansion
_{p}\oplus%
%TCIMACRO{\U{211a} }%
%BeginExpansion
\mathbb{Z}
%EndExpansion
)$. Then note that $\pi$ is surjective and $Ker(\pi)=p%
%TCIMACRO{\U{2124} }%
%BeginExpansion
\mathbb{Z}
%EndExpansion
\oplus(0)$ is not a prime submodule of $%
%TCIMACRO{\U{2124} }%
%BeginExpansion
\mathbb{Z}
%EndExpansion
\oplus%
%TCIMACRO{\U{211a} }%
%BeginExpansion
\mathbb{Z}
%EndExpansion
$. On the other hand, $%
%TCIMACRO{\U{2124} }%
%BeginExpansion
\mathbb{Z}
%EndExpansion
\oplus%
%TCIMACRO{\U{211a} }%
%BeginExpansion
\mathbb{Z}
%EndExpansion
$ is a torsion-free module, by Example \ref{prop1}, so is an annihilator
multiplication module. However, by Example \ref{ex1}, $%
%TCIMACRO{\U{2124} }%
%BeginExpansion
\mathbb{Z}
%EndExpansion
_{p}\oplus%
%TCIMACRO{\U{211a} }%
%BeginExpansion
\mathbb{Z}
%EndExpansion
$ is not an annihilator multiplication module.
\end{example}

According to \cite{AnFul}, a nonzero submodule $V$ of $E$ is defined as an \textit{essential submodule}, also referred to as a \textit{large submodule}, if for any submodule $K$ of $E$, the condition $K \cap V = (0)$ implies $K = (0)$. Furthermore, a submodule $V$ of $E$ is called a \textit{pure submodule} if $IE \cap V = IV$ for every ideal $I$ of $A\ $\cite{AnFul}.

\begin{corollary}
\label{csub}Let $E\ $be an $A$-module and $V\ $a submodule of $E.\ $The
following statements are satisfied.

(i)\ If $V\ $is a prime submodule of $E\ $such that $ann(V)=(V:E)$\ and
$E\ $is an annihilator multiplication module, then $E/V\ $is an annihilator
multiplication module.

(ii)\ If $E\ $is an annihilator multiplication module and $ann(E)=ann(V),\ $%
then $V\ $is an annihilator multiplication module.

(iii)\ If $E\ $is an annihilator multiplication module\ and $V\ $is a pure
submodule and essential in $E,\ $then $V\ $is an annihilator multiplication module.
\end{corollary}

\begin{proof}
$(i):\ $Consider the natural epimorphism $\pi:E\rightarrow E/V$ defined by
$\pi(e)=e+V$ for each $e\in E.\ $Also note that $V=Ker(\pi)\ $is a prime
submodule and $ann(E)=ann(E/V)=(V:E).\ $Then by Proposition \ref{phom}
$(ii)$,\ $E/V$ is an annihilator multiplication module.

$(ii):\ $Consider the injection $i:V\rightarrow E\ $defined by $i(v)=v\ $for
each $v\in V.\ $Then note that $i\ $is one to one and $ann(E)=ann(V).\ $Then
by Proposition \ref{phom} $(i)$,\ $V\ $is an annihilator multiplication module.

$(iii):\ $It is enough to show that $ann(V)=ann(E).\ $Let $a\in ann(V).\ $Then
we have $aV=(0).\ $Since $V\ $is a pure submodule, we conclude that $aV=aE\cap
V=(0).\ $As$\ V\ $is an essential submodule, we conclude that $aE=(0)$ which
implies that $ann(V)\subseteq ann(E).\ $Since the reverse inclusion always
holds, we have the equality $ann(V)=ann(E).$
\end{proof}

\begin{proposition}\label{pdirectsum}
Let $E$ be an $A$-module and $\{E_{i}\}_{i\in\Delta}$ be a family of submodules of $E$ such that $ann(E_{i})=ann(E_{j})$ for every $i\neq j$. Then
$E=%
%TCIMACRO{\tbigoplus \limits_{i\in\Delta}}%
%BeginExpansion
{\textstyle\bigoplus\limits_{i\in\Delta}}
%EndExpansion
E_{i}$ is an annihilator multiplication module if and only if $E_{i}$ is an
annihilator multiplication module for each $i\in\Delta$.
\end{proposition}

\begin{proof}
The if part follows from Corollary \ref{csub} (ii). For the only if part, let $E_{i}$ be an annihilator multiplication module for each $i\in\Delta$. Choose $e\in E$. Then we can write $e=e_{i_{1}}+e_{i_{2}}+\cdots+e_{i_{n}}$ for some
$i_{1},i_{2},\ldots,i_{n}\in\Delta$. Let $a\in ann(e)$. Then we conclude
that $a\left(  e_{i_{1}}+e_{i_{2}}+\cdots+e_{i_{n}}\right)  =0$ which implies
that $ae_{i_{1}}=-ae_{i_{2}}-ae_{i_{3}}-\cdots-ae_{i_{n}}\in E_{i_{1}}\cap%
%TCIMACRO{\tsum \limits_{j=2}^{n}}%
%BeginExpansion
{\textstyle\sum\limits_{j=2}^{n}}
%EndExpansion
E_{i_{j}}=(0)$. Thus we conclude that $a\in ann(e_{i_{1}})$. Similarly, we
have $a\in%
%TCIMACRO{\tbigcap \limits_{j=1}^{n}}%
%BeginExpansion
{\textstyle\bigcap\limits_{j=1}^{n}}
%EndExpansion
ann(e_{i_{j}})$. Then clearly we obtain $ann(e)=%
%TCIMACRO{\tbigcap \limits_{j=1}^{n}}%
%BeginExpansion
{\textstyle\bigcap\limits_{j=1}^{n}}
%EndExpansion
ann(e_{i_{j}})$. Since $E_{i_{j}}$ is an annihilator multiplication module for each $1\leq j\leq n$, there exist finitely generated ideals $I_{j}$ such that $ann(e_{i_{j}})=ann(I_{j}E_{j})$. On the other hand, we have $ann(I_{j}E_{j})=ann(I_{j}E)$ since $ann(E_{i})=ann(E_{j})$ for every $i\neq
j$ and $E$ is internal direct sum of the family $\{E_{i}\}_{i\in\Delta}$. This implies that
\begin{align*}
ann(e)  &  =%
%TCIMACRO{\tbigcap \limits_{j=1}^{n}}%
%BeginExpansion
{\textstyle\bigcap\limits_{j=1}^{n}}
%EndExpansion
ann(e_{i_{j}})=%
%TCIMACRO{\tbigcap \limits_{j=1}^{n}}%
%BeginExpansion
{\textstyle\bigcap\limits_{j=1}^{n}}
%EndExpansion
ann(I_{j}E_{j})\\
&  =%
%TCIMACRO{\tbigcap \limits_{j=1}^{n}}%
%BeginExpansion
{\textstyle\bigcap\limits_{j=1}^{n}}
%EndExpansion
ann(I_{j}E)=ann\left(  \left(
%TCIMACRO{\tsum \limits_{j=1}^{n}}%
%BeginExpansion
{\textstyle\sum\limits_{j=1}^{n}}
%EndExpansion
I_{j}\right)  E\right)  .
\end{align*}
As $%
%TCIMACRO{\tsum \limits_{j=1}^{n}}%
%BeginExpansion
{\textstyle\sum\limits_{j=1}^{n}}
%EndExpansion
I_{j}$ is finitely generated, $E$ is an annihilator multiplication module.
\end{proof}

The condition that $ann(E_{i})=ann(E_{j})$ for every $i\neq j$ is a necessary requirement in the preceding proposition. The subsequent example demonstrates the necessity of this condition.

\begin{example}
Let $p$ be a prime number. Since $%
%TCIMACRO{\U{211a} }%
%BeginExpansion
\mathbb{Q}
%EndExpansion
$ is a torsion-free $%
%TCIMACRO{\U{2124} }%
%BeginExpansion
\mathbb{Z}
%EndExpansion
$-module and $%
%TCIMACRO{\U{2124} }%
%BeginExpansion
\mathbb{Z}
%EndExpansion
_{p}$ is a multiplication $%
%TCIMACRO{\U{2124} }%
%BeginExpansion
\mathbb{Z}
%EndExpansion
$-module, so they are annihilator multiplication $%
%TCIMACRO{\U{2124} }%
%BeginExpansion
\mathbb{Z}
%EndExpansion
$-modules. Also, it is clear that $ann\left(
%TCIMACRO{\U{211a} }%
%BeginExpansion
\mathbb{Q}
%EndExpansion
)=(0\right)  \neq ann(%
%TCIMACRO{\U{2124} }%
%BeginExpansion
\mathbb{Z}
%EndExpansion
_{p})=p%
%TCIMACRO{\U{2124} }%
%BeginExpansion
\mathbb{Z}
%EndExpansion
.\ $However, $%
%TCIMACRO{\U{2124} }%
%BeginExpansion
\mathbb{Z}
%EndExpansion
$-module $E=%
%TCIMACRO{\U{211a} }%
%BeginExpansion
\mathbb{Q}
%EndExpansion
\oplus%
%TCIMACRO{\U{2124} }%
%BeginExpansion
\mathbb{Z}
%EndExpansion
_{p}$ is not an annihilator multiplication module since $ann(\overline{1})=p%
%TCIMACRO{\U{2124} }%
%BeginExpansion
\mathbb{Z}
%EndExpansion
\neq ann(IE)=ann(I)=(0)$ for any nonzero ideal $I$ of $\mathbb{Z}$.
\end{example}

Let $E$ be an $A$-module and $X$ an indeterminate over $A$. The notation $E[X]$ refers to the polynomial module over the polynomial ring $A[X]$. An $A$-module $E$ is defined as an \textit{Armendariz module} if, for every $e(X)=e_{0}+e_{1}X+e_{2}X^{2}+\cdots+e_{n}X^{n}$ in $E[X]$ and $f(X)=a_{0}+a_{1}X+a_{2}X^{2}+\cdots+a_{k}X^{k}$ in $A[X]$ such that $f(X)e(X)=0$, then $a_{i}e_{j}=0$ for every $1\leq i\leq k$ and $1\leq j\leq n$ \cite{Baser}. Prior to examining the stability of the annihilator multiplication property in polynomial modules, the following lemma is required.

\begin{lemma}
\label{lempol}Let $E$ be an Armendariz module. Then,

(i) $ann_{A[X]}(e(X))=\left[
%TCIMACRO{\tbigcap \limits_{i=0}^{n}}%
%BeginExpansion
{\textstyle\bigcap\limits_{i=0}^{n}}
%EndExpansion
ann_{A}(e_{i})\right]  \left[  X\right]  $ for every $e(X)=e_{0}+e_{1}%
X+\cdots+e_{n}X^{n}\in E[X].$

(ii)\ $ann_{A[X]}(p(X)E[X])=\left[
%TCIMACRO{\tbigcap \limits_{i=0}^{k}}%
%BeginExpansion
{\textstyle\bigcap\limits_{i=0}^{k}}
%EndExpansion
ann_{A}(a_{i}E)\right]  \left[  X\right]  \ $for every $p(X)=a_{0}%
+a_{1}X+\cdots+a_{k}X^{k}\in A[X].$
\end{lemma}

\begin{proof}
$\left(  i\right)  :\ $Follows from definition of Armendariz module.

$\left(  ii\right)  :\ $Suppose that $p(X)=a_{0}+a_{1}X+\cdots+a_{k}X^{k}\in
A[X].\ $Let $q(X)=b_{0}+b_{1}X+\cdots+b_{t}X^{t}\in A[X]$ such that $q(X)\in
ann_{A[X]}(p(X)E[X]).\ $Then we have $q(X)p(X)E[X]=0\ $which implies that
$q(X)\left(  p(X)e\right)  =0\ $for each $e\in E.\ $Since $p(X)e=a_{0}%
e+a_{1}eX+\cdots+a_{k}eX^{k}$ and $E\ $is an Armendariz module, we have
$b_{j}a_{i}e=0\ $for each $0\leq j\leq t$ and $0\leq i\leq k.\ $Then we
conclude that $b_{j}\in%
%TCIMACRO{\tbigcap \limits_{i=0}^{k}}%
%BeginExpansion
{\textstyle\bigcap\limits_{i=0}^{k}}
%EndExpansion
ann_{A}(a_{i}E)$ for every $0\leq j\leq t,\ $that is, $ann_{A[X]}%
(p(X)E[X])\subseteq\left[
%TCIMACRO{\tbigcap \limits_{i=0}^{k}}%
%BeginExpansion
{\textstyle\bigcap\limits_{i=0}^{k}}
%EndExpansion
ann_{A}(a_{i}E)\right]  \left[  X\right]  .$\ For the reverse inclusion, let
$q(X)=b_{0}+b_{1}X+\cdots+b_{t}X^{t}\in\left[
%TCIMACRO{\tbigcap \limits_{i=0}^{k}}%
%BeginExpansion
{\textstyle\bigcap\limits_{i=0}^{k}}
%EndExpansion
ann_{A}(a_{i}E)\right]  \left[  X\right]  .\ $Then we have $b_{j}\in%
%TCIMACRO{\tbigcap \limits_{i=0}^{k}}%
%BeginExpansion
{\textstyle\bigcap\limits_{i=0}^{k}}
%EndExpansion
ann_{A}(a_{i}E)$ which implies that $b_{j}a_{i}e=0$ for every $e\in E,\ 0\leq
j\leq t$ and $0\leq i\leq k.\ $Choose $h(X)=e_{0}+e_{1}X+\cdots+e_{s}X^{s}\in
E[X]$.\ Then note that $p(X)h(X)=a_{0}e_{0}+(a_{1}e_{0}+a_{0}e_{1}%
)X+\cdots+a_{k}e_{s}X^{k+s}.\ $Since $b_{j}a_{i}e=0$ for every
$e\in E,\ 0\leq j\leq t$ and $0\leq i\leq k,\ $we have $q(X)p(X)h(X)=0.\ $This
implies that $q(X)p(X)E[X]=0,\ $that is, $q(X)\in ann_{A[X]}(p(X)E[X]).\ $%
Thus, we conclude that $\left[
%TCIMACRO{\tbigcap \limits_{i=0}^{k}}%
%BeginExpansion
{\textstyle\bigcap\limits_{i=0}^{k}}
%EndExpansion
ann_{A}(a_{i}E)\right]  \left[  X\right]  \subseteq ann_{A[X]}(p(X)E[X]).$
\end{proof}

\begin{proposition}\label{ppol}
Let $E\ $be an Armendariz $A$-module. Then $E\ $is an annihilator multiplication $A$-module if and only if $E[X]$ is an annihilator multiplication $A[X]$-module.
\end{proposition}

\begin{proof}
$\left(  \Leftarrow\right)  :\ $Suppose that $E[X]$ is an annihilator
multiplication $A[X]$-module. Let $e_{0}\in E$ and put $e(X):=e_{0}$ which is
a constant polynomial. Since $E[X]$ is an annihilator multiplication
$A[X]$-module, there exist $p_{1}(X),p_{2}(X),\ldots,p_{k}(X)\in A[X]$ such
that
\begin{align*}
ann_{A[X]}(e(X))  &  =\left[  ann_{A}(e_{0})\right]  \left[  X\right]
=ann_{A[X]}\left(
%TCIMACRO{\tsum \limits_{i=1}^{k}}%
%BeginExpansion
{\textstyle\sum\limits_{i=1}^{k}}
%EndExpansion
p_{i}(X)E[X]\right) \\
&  =%
%TCIMACRO{\tbigcap \limits_{i=1}^{k}}%
%BeginExpansion
{\textstyle\bigcap\limits_{i=1}^{k}}
%EndExpansion
ann_{A[X]}(p_{i}(X)E[X]).
\end{align*}
Let $p_{i}(X)=a_{0,i}+a_{1,i}X+a_{2,i}X^{2}+\cdots+a_{t,i}X^{t}.\ $Then by
Lemma \ref{lempol} (ii), we have $ann_{A[X]}(p_{i}(X)E[X])=\left[
%TCIMACRO{\tbigcap \limits_{j=0}^{t}}%
%BeginExpansion
{\textstyle\bigcap\limits_{j=0}^{t}}
%EndExpansion
ann_{A}(a_{j,i}E)\right]  \left[  X\right]  .\ $This implies that
\begin{align*}
ann_{A[X]}(e(X))  &  =\left[  ann_{A}(e_{0})\right]  \left[  X\right]  =%
%TCIMACRO{\tbigcap \limits_{i=1}^{k}}%
%BeginExpansion
{\textstyle\bigcap\limits_{i=1}^{k}}
%EndExpansion
ann_{A[X]}(p_{i}(X)E[X])\\
&  =\left[
%TCIMACRO{\tbigcap \limits_{i=1}^{k}}%
%BeginExpansion
{\textstyle\bigcap\limits_{i=1}^{k}}
%EndExpansion%
%TCIMACRO{\tbigcap \limits_{j=0}^{t}}%
%BeginExpansion
{\textstyle\bigcap\limits_{j=0}^{t}}
%EndExpansion
ann_{A}(a_{j,i}E)\right]  \left[  X\right]  .
\end{align*}
Thus we obtain,
\begin{align*}
ann_{A}(e_{0})  &  =%
%TCIMACRO{\tbigcap \limits_{i=1}^{k}}%
%BeginExpansion
{\textstyle\bigcap\limits_{i=1}^{k}}
%EndExpansion%
%TCIMACRO{\tbigcap \limits_{j=0}^{t}}%
%BeginExpansion
{\textstyle\bigcap\limits_{j=0}^{t}}
%EndExpansion
ann_{A}(a_{j,i}E)\\
&  =ann_{A}\left(  \left(
%TCIMACRO{\dsum \limits_{\genfrac{.}{.}{0pt}{}{1\leq i\leq k}{0\leq j\leq t}}}%
%BeginExpansion
{\displaystyle\sum\limits_{\genfrac{.}{.}{0pt}{}{1\leq i\leq k}{0\leq j\leq
t}}}
%EndExpansion
Aa_{j,i}\right)  E\right)  .
\end{align*}
Hence, $E\ $is an annihilator multiplication $A$-module since $%
%TCIMACRO{\dsum \limits_{\genfrac{.}{.}{0pt}{}{1\leq i\leq k}{0\leq j\leq t}}}%
%BeginExpansion
{\displaystyle\sum\limits_{\genfrac{.}{.}{0pt}{}{1\leq i\leq k}{0\leq j\leq
t}}}
%EndExpansion
Aa_{j,i}$ is a finitely generated ideal.

$\left(  \Rightarrow\right)  :$\ Let $E\ $be an annihilator multiplication
$A$-module and $e(X)=e_{0}+e_{1}X+e_{2}X^{2}+\cdots+e_{n}X^{n}\in E[X].\ $Then
by Lemma \ref{lempol} (i),
\begin{align*}
ann_{A[X]}\left(  e(X)\right)   &  =\left[
%TCIMACRO{\tbigcap \limits_{i=0}^{n}}%
%BeginExpansion
{\textstyle\bigcap\limits_{i=0}^{n}}
%EndExpansion
ann_{A}(e_{i})\right]  \left[  X\right] \\
&  =\left[  ann_{A}(Ae_{0}+Ae_{1}+\cdots+Ae_{n})\right]  \left[  X\right]  .
\end{align*}
Since $Ae_{0}+Ae_{1}+\cdots+Ae_{n}$ is a finitely generated submodule of
$E\ $and $E\ $is an annihilator multiplication module, by Proposition
\ref{prop2}, there exists a finitely generated ideal $I$ of $A\ $such that
\[
ann_{A}(Ae_{0}+Ae_{1}+\cdots+Ae_{n})=ann_{A}(IE).
\]
This gives
\begin{align*}
ann_{A[X]}\left(  e(X)\right)   &  =\left[  ann_{A}(IE)\right]  \left[
X\right]  =ann_{A[X]}\left(  (IE)[X]\right) \\
&  =ann_{A[X]}(I[X]E[X]).
\end{align*}

Since $I\ $is finitely generated in $A,$ $I[X]$ is finitely generated in
$A[X].\ $Thus, it follows that $E[X]\ $is an annihilator multiplication $A[X]$-module.
\end{proof}

Let $E$ be an $A$-module. \textit{The trivial extension} $A\propto E=A\oplus E$ of $A$-module $E$ is a commutative ring with a nonzero identity $(1,0)$ under componentwise addition and following multiplication: $(a,e)(b,v)=(ab,av+be)$ for each $(a,e),(b,v)\in A\propto E$ (See, \cite{AndersonWinders} and \cite{Huckaba}). For an ideal $I$ of $A$ and a submodule $V$ of $E$, $I\propto V$ is an ideal of $A\propto E$ if and only if $IE\subseteq V$ \cite[Theorem 3.1]{AndersonWinders}. Now, we are ready to determine certain annihilator multiplication/multiplication ideals in trivial ring extension $A\propto E$.

\begin{theorem}\label{tri1}
Let $E$ be a finitely generated $A$-module and $V$ a submodule of $E$. The following statements are equivalent.
\begin{itemize}
\item [(i)] $V$ is an annihilator multiplication module.
\item [(ii)] $0\propto V$ is an annihilator multiplication ideal of $A\propto E$.
\end{itemize}
\end{theorem}

\begin{proof}
$\left(i\right)\Leftarrow\left(ii\right):$ Let $V$ be an annihilator multiplication module and choose $(0,v)\in\ 0\propto V$. Since $V$ is an annihilator multiplication module, there exists a finitely generated ideal $I$ of $A$ such that $ann(v)=ann(IV)$. This gives $ann(0,v)=ann(v)\propto E=ann(IV)\propto E=ann((I\propto E)(0\propto V))$. As $I$ and $E$ are finitely generated, it follows that $I\propto E$ is a finitely generated ideal in $A\propto E$. Thus, $0\propto V$ is an annihilator multiplication ideal in $A\propto E$.\\
$\left(ii\right)\Leftarrow\left(i\right):$ Let $0\propto V$ be an annihilator multiplication ideal of $A\propto E$. Choose $v\in V$, and by the assumption, note that $ann(0,v)=ann(\frak{J}(0\propto V))$ for some finitely generated ideal $\frak{J}$ of $A\propto E$. As $E$ is finitely generated, it follows that $\frak{J}+(0\propto E)$ is also a finitely generated ideal in $A\propto E$. Since $(\frak{J}+0\propto E)(0\propto V)=\frak{J}(0\propto V)+(0\propto E)(0\propto V)=\frak{J}(0\propto V)$, we may assume that $\frak{J}$ contains $0\propto E$. In this case, by \cite[Theorem 3.1]{AndersonWinders}, $\frak{J}=I\propto E$ and $I$ is a finitely generated ideal in $A$. Thus we have $\frak{J}(0\propto V)=0\propto IV$, and this gives $ann(0,v)=ann(v)\propto E=ann(\frak{J}(0\propto V))=ann(0\propto IV)=ann(IV)\propto E$. Therefore, we conclude that $ann(v)=ann(IE)$ which completes the proof.
\end{proof}

Recall from \cite{Sharp} that an $A$-module $E$ is said to be a \textit{simple module} if $Ae=E$ for every $0\neq e\in E$. Now, we are ready to characterize simple modules in terms of certain multiplication ideals of the trivial extension $A\propto E$.

\begin{theorem}\label{tri2}
Let $A$ be a local ring with unique maximal ideal $\frak{m}$ and $E$ be an $A$-module such that $\frak{m}E=0$. Suppose that $V$ is a submodule of $E$. Then $0\propto V$ is a multiplication ideal of $A$ if and only if $V$ is a simple module. 
\end{theorem}

\begin{proof}
$(\Rightarrow):$ Suppose that $0\propto V$ is a multiplication ideal of $A\propto E$. Choose $0\neq v\in V$. Since $0\propto V$ is a multiplication ideal and $(0,v)\in 0\propto V$. By \cite[Lemma 3.5]{Cho}, there exists $(a_1,v_1),(a_2,v_2),\ldots, (a_n,v_n)\in A\propto E$ such that $A\propto E (0,v)=0\propto Av=\sum_{i=1}^{n}(a_i,v_i) 0\propto V$. If $a_i$'s are nonunits of $A$, then $a_i\in \frak{m}$. Then by the assumption, $\sum_{i=1}^{n}(a_i,v_i) 0\propto V=0\propto 0=0\propto Av$ which implies that $Av=0$, a contradiction. Then there exists $i=1,2,\ldots,n$ such that $a_i$ is a unit of $A$. In this case, we have $\sum_{i=1}^{n}(a_i,v_i) 0\propto V=0\propto V=A\propto E(0,v)=0\propto Av$ which implies that $Av=V$, that is $V$ is a simple module.\\
$(\Leftarrow):$ Assume that $V$ is a simple module. Choose a nonzero element $(0,v)\in 0\propto V$. Since $V$ is a simple module and $v\neq 0$, we have $Av=V$. This implies that $A\propto E (0,v)=0\propto Av=0\propto V= (A\propto E)(0\propto V)$, that is $0\propto V$ is a multiplication ideal of $A\propto E$.
\end{proof}

\begin{theorem}\label{tri3}
Let $A$ be a local ring with unique maximal ideal $\frak{m}$ and $E$ be an $A$-module such that $\frak{m}E=0$. For a proper ideal $I$ of $A$, $I\propto E$ is an annihilator multiplication ideal of $A\propto E$ if and only if $I$ is an annihilator multiplication ideal of $A$.
\end{theorem}

\begin{proof}
$(\Leftarrow):$ Suppose that $I$ is an annihilator multiplication ideal of $A$ and choose $(a,v)\in I\propto E$. Then $a\in I\subseteq \frak{m}$ which implies that $ann(a,v)=ann(a)\propto E$. Since $I$ is an annihilator multiplication ideal and $a\in I$, there exists a finitely generated ideal $J=\sum_{i=1}^{k}Ax_i$ of $A$ such that $ann(a)=ann(JI)=\bigcap_{i=1}^{k}ann(x_i I)$. Now, we have two cases. \textbf{Case 1:} Assume that $J=A$. In this case, we obtain $ann(a)=ann(I)$ which implies that $ann(a,v)=ann(a)\propto E=ann(I)\propto E=ann((1,0)I\propto E)$. Thus, $I\propto E$ is an annihilator multiplication ideal of $A\propto E$. \textbf{Case 2:} Suppose that $J$ is a proper ideal of $A$. In this case, $x_i$ is a nonunit and hence an element of $\frak{m}$ for each $i=1,2,\ldots,k$. Since $\frak{m}E=0$, it is easy to see that $(x_i,0)I\propto E=x_iI\propto 0$ which implies that $ann(x_iI)\propto E=ann((x_i,0)I\propto E)$. Then we conclude that $ann(a,v)=ann(a)\propto E=(\bigcap_{i=1}^{k} ann(x_iI))\propto E=\bigcap_{i=1}^{k}(ann(x_iI)\propto E)=\bigcap_{i=1}^{k}ann((x_i,0)I\propto E)$. Now, put $\frak{J}=\sum_{i=1}^{k}(x_i,0)A\propto E$. Then $\frak{J}$ is a finitely generated ideal in $A\propto E$ and it is clear that $ann(a,v)=ann(\frak{J}(I\propto E))$ which completes the proof.\\
$(\Rightarrow):$ Choose an element $a\in I$. Since $I\propto E$ is an annihilator multiplication ideal and $\frak{m}E=0$, we have $ann(a,0_E)=ann(a)\propto E=ann(\frak{J}(I\propto E))$ for some finitely generated ideal $\frak{J}$ of $A\propto E$. Now, we have two cases. \textbf{Case 1:} Let $\frak{J}=A\propto E$. Then we have $ann(a,0_E)=ann(a)\propto E=ann(I\propto E)=ann(I)\propto E$ which implies that $ann(a)=ann(I)$. As $A$ is principal ideal and $AI=I$, it follows that $I$ is an annihilator multiplication ideal of $A$. \textbf{Case 2:} Let $\frak{J}$ be a proper ideal of $A\propto E$. Then for any $(b,v)\in \frak{J}$, $b\in \frak{m}$ because $b\in A$ is nonunit and $\frak{m}$ is the unique maximal ideal of $A$. This gives that $(b,v)(0,z)=(0,bz)=(0,0)$ for any $z\in E$, and thus we obtain $\frak{J}(0\propto E)=0\propto 0$. Then we conclude that $\frak{J}[(I\propto E)+(0\propto E)]=\frak{J}(I\propto E)$. Thus, we may assume that $\frak{J}$ contains $0\propto E$. Then by \cite[Theorem 3.1]{AndersonWinders}, we can write $\frak{J}=K\propto E$ for some finitely generated proper ideal $K$ of $A$. This implies that $\frak{J}(I\propto E) =(K\propto E)(I\propto E)=KI\propto 0$. Then we have $ann(a,0_E)=ann(a)\propto E=ann(\frak{J}(I\propto E))=ann(KI\propto 0)=ann(KI)\propto E$ which implies that $ann(a)=ann(KI)$ for some finitely generated ideal $K$ of $A$, that is $I$ is an annihilator multiplication ideal.\\
\end{proof}

The following example provides a construction of an annihilator multiplication ring that is not a multiplication ring.

\begin{example}\label{construction}
Let $E$ be a vector space over a field $k$ with $dim_{k}(E)=n\geq 2$. Consider the trivial extension $R=k\propto E$. Then $R$ is an annihilator multiplication ring which is not a multiplication ring.
\end{example}

\begin{proof}
Suppose that $R=k\propto E$, where $E$ is a vector space over the field $k$ with $dim_{k}(E)=n\geq 2$. Then by \cite[Corollary 3.4]{AndersonWinders}, we know that every proper ideal of $R$ has the form $0\propto V$ for some subspace $V$ of $E$. Since $E$ is a vector space, its all subspaces are torsion-free, so by Example \ref{prop1}, $V$ is an annihilator multiplication module. Then by Theorem \ref{tri1}, $0\propto V$ is an annihilator multiplication ideal of $R$. Thus, $R$ is an annihilator multiplication ring. On the other hand, put $V=E$. Since $E$ is not a simple module, by Theorem \ref{tri2}, $0\propto V$ is not a multiplication ideal of $R$, and hence $R$ can not be a multiplication ring.
\end{proof}

\section{Characterizations of some special rings/modules}
This section is devoted to characterizing several significant classes of rings and modules, including torsion-free, multiplication, uniserial, injective modules, classical prime/classical 1-absorbing prime submodules, and Noetherian vn-regular rings. Furthermore, we investigate the relationship between the associated primes of the module, $Ass_{A}(E)$, and the associated primes of the underlying ring, $Ass(A)$, within the framework of annihilator multiplication modules.

\begin{theorem}\label{ttorsion}
Let $E\ $be an $A$-module. Then $E\ $is a torsion-free module if and only if
$E\ $is an annihilator multiplication module and faithful module over the
domain $A.$
\end{theorem}

\begin{proof}
$\left(  \Rightarrow\right)  :\ $If $E\ $is a torsion-free module, by
Example \ref{prop1} $(iii)$, $E\ $is an annihilator multiplication module.
Also note that if $E\ $is torsion-free, then clearly $A\ $is a domain and
$E\ $is a faithful module.\\ $\left(  \Leftarrow\right)  :\ $Suppose that
$E\ $is an annihilator multiplication module and faithful module over the
domain $A.$ Let $0\neq e\in E.\ $Now, we will show that $ann(e)=(0).\ $Since
$E\ $is an annihilator multiplication module, there exists a finitely
generated ideal $I\ $of $A\ $such that $ann(e)=ann(IE).\ $As $I\ $is a
finitely generated ideal, we can write $I=%
%TCIMACRO{\tsum \limits_{i=1}^{n}}%
%BeginExpansion
{\textstyle\sum\limits_{i=1}^{n}}
%EndExpansion
Ra_{i}\ $for some $a_{i}\in I.\ $Since $E\ $is a faithful module, we have
$ann(IE)=ann(I)=%
%TCIMACRO{\tbigcap \limits_{i=1}^{n}}%
%BeginExpansion
{\textstyle\bigcap\limits_{i=1}^{n}}
%EndExpansion
ann(a_{i})=ann(e).\ $As $e\ $is not zero, $I\ $can not be a zero ideal. Since
$A\ $is a domain, $ann(a_{i})=(0)\ $for every $i=1,2,\ldots,n.\ $Then we have
$ann(e)=(0)$ which completes the proof.
\end{proof}

Consider an $A$-module $E$. Given any submodule $V$ of $E$ and a nonempty subset $J$ of $A$, we define the residual of $V$ by $J$ as $(V:_{E}J)=\{ e\in E : Je\subseteq V \}$. When $V$ is the zero submodule, we use $ann_{E}(J)$ to represent $((0):_{E}J)$. As described in \cite{TorFar}, $E$ is called a \textit{comultiplication module} if, for every submodule $V$ of $E$, there is an ideal $I$ of $R$ such that $V=ann_{E}(I)$. In fact, $E$ is a comultiplication module precisely when $V=ann_{E}(ann(V))$. For further exploration of comultiplication modules, see \cite{ShaSmi} and \cite{YilTeKo}.

\begin{proposition}\label{pmult}
Let $E\ $be a comultiplication module. Then $E\ $is a multiplication module if
and only if $E\ $is an annihilator multiplication module.
\end{proposition}

\begin{proof}
$\left(  \Rightarrow\right)  :\ $Follows from Example \ref{prop1}.\\
$\left(  \Leftarrow\right)  :$\ Let $E\ $be a comultiplication and
annihilator multiplication module. Then by Proposition \ref{prop2}\ $(ii)$,
for every submodule $V\ $of $E$ there exists an ideal $I\ $of $A$ such that
$ann(V)=ann(IE).\ $As $E\ $is a comultiplication module, we have
$V=ann_{E}(ann(V))=ann_{E}(ann(IE))=IE$ which completes the proof.
\end{proof}

According to \cite{Beh}, a proper submodule $V$ of $E$ is defined as a \textit{classical prime submodule} if, whenever $abe \in V$ for some $a, b \in A$ and $e \in E$, it follows that $ae \in V$ or $be \in V$. Furthermore, $V$ is called a \textit{classical 1-absorbing prime} if, for some nonunits $a, b, c \in A$ and $e \in E$, the condition $abce \in V$ implies that $abe \in V$ or $ce \in V$ \cite{Zeynep}. One can easily see that every classical prime is also a classical 1-absorbing prime; however, the converse does not generally hold. The following example illustrates this distinction.

\begin{example}\label{exampleclassical}
	Consider the $k[[X]]$-module $E=k[[X]]/(X^2)$ where $k$ is a field and $X$ is an indeterminate over $k$. Let $V$ be the zero submodule of $E$. Then $V$ is not a classical prime submodule, since $X^2\overline{1}=\overline{0}\in V$ and $X\overline{1}=\overline{X}\notin V$. Let $abc\overline{e}=\overline{0}$ for some nonunits $a,b,c\in\ k[[X]]$ and $\overline{e}\in E$.  Since $k[[X]]$ is a local ring with a unique maximal ideal $(X)$,  we have $a=Xf$, $b=Xg$ and $c=Xh$ for some $f,g,h\in k[[X]]$. This gives $ab\overline{e}=X^2fg\overline{e}=\overline{0}\in V$, and so $V$ is a classical 1-absorbing prime submodule of $E$. 
\end{example}

\begin{proposition}\label{pclass}
	Let $E$ be an annihilator multiplication module and $ann(e)\neq ann(E)$ for every $e\in E$. Then the zero submodule of $E$ is a classical prime submodule if and only if it is a classical 1-absorbing prime.
\end{proposition}

\begin{proof}
	The if part follows from \cite[Proposition 1]{Zeynep}. For the only if part, assume that the zero submodule is a classical 1-absorbing prime submodule. Let $abe=0$ for some $a,b\in A$ and $e\in E$. We may assume that $a,b$ are nonunits of $A$. Since $E$ is an annihilator multiplication module, there exists a finitely generated ideal $I$ of $A$ such that $ann(e)=ann(IE)$. Since $ann(e)\neq ann(E)$, $I$ is a proper ideal of $A$. This gives $abIE=0$. If $be=0$, then we are done. So, assume that $be\neq 0$. Now, we will show that $aIE=0$. Choose an arbitrary $z\in E$ and assume that $aIz\neq 0$. Since $abIz=0$, by \cite[Theorem 2]{Zeynep}, we have $bz=0$. 
	On the other hand, since $abIe=0$ and $be\neq 0$, we have $aIe=0$. As $abI(z+e)=0$ and $b(z+e)=be\neq 0$, we conclude that $aI(z+e)=aIz=0$ which is a contradiction.
	Thus, we conclude that $aIE=0$, that is, $a\in ann(IE)=ann(e)$. This gives that $ae=0$, as required.
\end{proof}

\begin{remark}
	In the previous proposition, the condition $ann(e)\neq ann(E)$ is necessary. For example, consider the $k[[X]]$-module $E=k[[X]]/(X^2)$, where $k$ is a field and $X$ is an indeterminate over $k$. Then $E$ is clearly a multiplication module, so by Example \ref{prop1}, it is an annihilator multiplication module. However, note that $ann(\overline{1})=(X^2)=ann(E)$. Also, by Example \ref{exampleclassical}, $V=(\overline{0})$ is a classical 1-absorbing prime submodule which is not a classical prime submodule. 
\end{remark}

Recall from \cite{1ABS} that a proper ideal $I$ of $A\ $is said to be a \textit{1-absorbing prime ideal} if whenever $abc\in I$ for some nonunits $a,b,c\in A$, then $ab\in I$ or $c\in I$.

\begin{proposition}\label{1-abs}
	Let $E$ be an annihilator multiplication module over which $ann(E)\ $is a
	1-absorbing prime ideal of $A.\ $Then, for every submodule $V$ of $E$, either
	$ann(V)=ann(E)$ or $ann(V)$ is a prime ideal of $A$. In this case,
	$ann(V)$ is a 1-absorbing prime ideal of $A$. Furthermore, $\left\{
	ann(V):\ V\ \text{is a submodule of }E\right\}$ is totally ordered by inclusion.
\end{proposition}

\begin{proof}
	Suppose that $E$ is an annihilator multiplication module and $V$ is a
	submodule of $E$ such that $ann(V)\neq ann(E)$. Then there exists a proper ideal $I$ of $A$ such that $ann(V)=ann(IE)$ by Proposition \ref{prop2} (ii). Let $ab\in ann(V)$ for some
	$a,b\in A$. Then we have $abV=aIbE=0$ which implies that $aIb\subseteq ann(E)$. If $a$ or $b$ is unit, then we are done. So assume that $a,b$ are nonunits. Since $ann(E)$ is a 1-absorbing prime ideal, we conclude that $aI\subseteq ann(E)$ or $b\in ann(E)$. Then we have $a\in ann(IE)=ann(V)$ or $b\in ann(E)\subseteq ann(V)$. Thus, $ann(V)$ is a prime ideal of $A$, so is 1-absorbing prime. Let $V,K$ be two submodules of $E$ such that
	$ann(V)\neq ann(E)$ and $ann(K)\neq ann(E)$. Then by above, either we have $ann(V+K)=ann(E)$ or $ann(V+K)$ is a prime ideal of $A$. Now we have two cases. \textbf{Case 1: }Let $ann(V+K)=ann(E)$. Since $ann(V+K)=ann(V)\cap
	ann(K)$, we have $ann(V)\cap ann(K)=ann(E)$. By assumption, we get $ann(K)$ and $ann(V)$ are prime ideals of $A$ which implies that $\sqrt{ann(E)}%
	=\sqrt{ann(V)}\cap\sqrt{ann(K)}=ann(V)\cap ann(K)=ann(E)$. Since $ann(E)$ is a 1-absorbing prime ideal and semiprime, so is prime ideal. Then we have
	either $ann(V)\subseteq ann(K)$ or $ann(K)\subseteq ann(V)$. \textbf{Case 2:} Let $ann(V+K)\neq ann(E)$. In this case, $ann(V+K)=ann(V)\cap ann(K)$ is a prime ideal which implies that either $ann(V)\subseteq ann(K)$ or
	$ann(K)\subseteq ann(V)$. Hence, $\left\{  ann(V):\ V\ \text{is a submodule of
	}E\right\}$ is totally ordered by inclusion.
\end{proof}

Let $E$ be an $A$-module. Recall from \cite{FacSal} that $E$ is said to be a \textit{uniserial module} if its each two submodules are comparable, that is, $V\subseteq W$ or $W\subseteq V$ for every two submodules $V,W$ of $E$. For more details on uniserial modules, we refer \cite{FacSal} and \cite{Tekir} to the reader. 

\begin{corollary}\label{uniserial}
	Let $E$ be an annihilator multiplication module and a comultiplication module such that $Ann(E)$ is a 1-absorbing prime ideal of $A$. Then $E$ is a uniserial module. 
\end{corollary}

\begin{proof}
	Suppose that $E$ is an annihilator multiplication module and a comultiplication module. Let $V,W$ be two submodules of $E$. Then by Proposition \ref{1-abs}, we have either $ann(V)\subseteq ann(W)$ or $ann(W)\subseteq ann(V)$. Without loss of generality, we may assume that $ann(V)\subseteq ann(W)$. Since $E$ is a comultiplication module, we have $W=ann_E(ann(W))\subseteq ann_E(ann(V))=V$, that is $E$ is a uniserial module.
\end{proof}

According to \cite{Annin}, an $A$-module $E$ is defined as a \textit{prime module} if the zero submodule of $E$ is prime, or equivalently, $ann(V) = ann(E)$ for every nonzero submodule $V$ of $E$. Additionally, a submodule $V$ of $E$ is defined to be a \textit{second submodule} if, for each $a \in A$, either $aV = (0)$ or $aV = V$. In particular, $E$ is called a \textit{second module} if it is a second submodule of itself \cite{Yas}. An $A$-module $E$ is defined as \textit{injective} if, whenever
$f: L \rightarrow K$ is an injective homomorphism and $g: L \rightarrow E$ is an
arbitrary homomorphism, there exists a homomorphism $h: K \rightarrow
E$ such that $fh = g$ \cite{AnFul}. Furthermore, \cite[Theorem 2.3]{Yas} demonstrates
that a prime module $E$ is a second module if and only if it is an
injective module over the ring $A/ann(E)$.

\begin{proposition}\label{pinj}
Let $E\ $be a second module and annihilator multiplication module. Then
$E\ $is a prime module and thus every submodule $V\ $of $E\ $is an annihilator
multiplication module. In this case, $E\ $is an injective $A/ann(E)$-module. In particular, if $E$ is a faithful module, then $E$ is an injective $A$-module. Furthermore, nonzero pure submodules and second submodules of $E\ $coincide.
\end{proposition}

\begin{proof}
Suppose that $E\ $is a second submodule and annihilator multiplication module.
Now, we will show that $ann(e)=ann(E)\ $for every $0\neq e\in E.\ $Choose
$0\neq e\in E.\ $As $E\ $is an annihilator multiplication module, there exists
a finitely generated ideal $I\ $of $A\ $such that $ann(e)=ann(IE).\ $Since
$I\ $is finitely generated, there exists $a_{1},a_{2},\ldots,a_{n}\in I\ $such
that $I=%
%TCIMACRO{\tsum \limits_{i=1}^{n}}%
%BeginExpansion
{\textstyle\sum\limits_{i=1}^{n}}
%EndExpansion
Aa_{i}.\ $This implies that $ann(e)=ann(IE)=%
%TCIMACRO{\tbigcap \limits_{i=1}^{n}}%
%BeginExpansion
{\textstyle\bigcap\limits_{i=1}^{n}}
%EndExpansion
ann(a_{i}E).\ $As $e\neq0,\ $there exists $t\in\{1,2,\ldots,n\}$ such that
$a_{t}E\neq0.\ $Otherwise, we would have $ann(e)=A\ $which implies that $e=0,$
a contradiction. Since $E\ $is a second module, $a_{t}E=E$ which yields that
\begin{align*}
ann(E)  &  =ann(a_{t}E)\subseteq ann(e)\\
&  =%
%TCIMACRO{\tbigcap \limits_{i=1}^{n}}%
%BeginExpansion
{\textstyle\bigcap\limits_{i=1}^{n}}
%EndExpansion
ann(a_{i}E)\subseteq ann(a_{t}E)=ann(E).
\end{align*}
Then we have $ann(e)=ann(E)$. This implies that $ann(V)=ann(E)$ for every nonzero submodule $V$ of $E$. Thus, $E\ $is a prime
module. Furthermore, by Corollary \ref{csub} $(ii)$, $V\ $is an annihilator
multiplication module for every nonzero submodule $V\ $of $E.\ $Since $E$ is prime and second module, by \cite[Theorem 2.3]{Yas}, $E\ $is an injective
$A/ann(E)$-module. For the rest, assume that $V\ $is a second submodule of
$E.\ $Let $J\ $be an ideal of $A.\ $Since $V\ $is a second submodule, either
$JV=V\ $or $JV=(0).\ $If $JV=(0),\ $then $J\subseteq ann(V)=ann(E)\ $which
implies that $JE\cap V=(0)=JV.\ $Now, assume that $JV=V.\ $Then we have
$JE\neq(0)\ $since $ann(V)=ann(E)$. As $E\ $is a second module, we have
$JE=E\ $which implies that $JE\cap V=V=JV.\ $Thus, $V\ $is a pure submodule.
For the converse, assume that $V\ $is a nonzero pure submodule of $E\ $and
$a\in A.\ $Then we have $aV=aE\cap V.\ $Without loss of generality, we may
assume that $aV\neq0.\ $As $ann(V)=ann(E)$ and $E\ $is second$,\ $we have
$aE\neq(0)\ $and so $aE=E.\ $This gives $aE\cap V=V=aV,\ $that is, $V\ $is a
second submodule of $E.\ $
\end{proof}

It is well known that a ring $A$ is von Neumann regular if and only if $A$ is reduced (i.e, a ring without a nonzero nilpotent) and every prime ideal is maximal. Now, we characterize Noetherian von Neumann regular rings in terms of annihilator multiplication modules.

\begin{theorem}\label{tvon}
Let $A$ be a ring. The following statements are equivalent.

(i)\ $A\ $is a reduced ring and every faithful $A$-module is an annihilator
multiplication module.

(ii)\ $A\ $is a principal ideal ring and von Neumann regular ring.

(iii) $A$ is a Noetherian von Neumann regular ring.

(iv) $A\ $is a direct product of finitely many fields.
\end{theorem}

\begin{proof}
$(i)\Rightarrow(ii):\ $Suppose that $A\ $is a reduced ring and every faithful
$A$-module is an annihilator multiplication module. Let $I$ be an ideal of
$A$, and put $E=A\oplus\left(  A/I\right)  .$\ Note that $E\ $is a faithful module, by the assumption, $E\ $is an annihilator multiplication
module. This implies that $ann((0,\overline{1}))=I=ann(JE)=ann(J)$ for some
finitely generated ideal $J$ of $A.\ $Then there exist $x_{1},x_{2}%
,\ldots,x_{n}\in J$ such that $J=%
%TCIMACRO{\tsum \limits_{i=1}^{n}}%
%BeginExpansion
{\textstyle\sum\limits_{i=1}^{n}}
%EndExpansion
Ax_{i}.\ $This implies that $I=%
%TCIMACRO{\tbigcap \limits_{i=1}^{n}}%
%BeginExpansion
{\textstyle\bigcap\limits_{i=1}^{n}}
%EndExpansion
ann(x_{i}).\ $Now, take a prime ideal $P$ of $A.\ $Then we have $P=%
%TCIMACRO{\tbigcap \limits_{i=1}^{n}}%
%BeginExpansion
{\textstyle\bigcap\limits_{i=1}^{n}}
%EndExpansion
ann(y_{i})$ for some $y_{1},y_{2},\ldots,y_{n}\in A.\ $Since $P\ $is prime, we
conclude that $P=ann(y_{i})$ for some $y_{i}\in A.\ $As $A\ $is a reduced
ring, by \cite[Lemma 2.8]{HenJe}, $P\ $is a minimal prime ideal of $A.\ $Thus,
every prime ideal of $A\ $is maximal, so $A\ $is a von Neumann regular ring.
Now, we will show that $A\ $is a principal ideal ring. Let $I\ $be an ideal of
$A.\ $Then by above, $I=%
%TCIMACRO{\tbigcap \limits_{i=1}^{n}}%
%BeginExpansion
{\textstyle\bigcap\limits_{i=1}^{n}}
%EndExpansion
ann(x_{i}).\ $Since $A\ $is a von Neumann regular ring, every principal ideal
$(x_{i})=(b_{i})\ $is generated by an idempotent $b_{i}\in A.\ $Thus, we
conclude that $ann(x_{i})=(1-b_{i}),\ $where $1-b_{i}\ $is again an idempotent
of $A.\ $Since $(b)\cap(c)=(bc)\ $for every two idempotents $b,c\in A,\ $we
obtain
\begin{align*}
I  &  =%
%TCIMACRO{\tbigcap \limits_{i=1}^{n}}%
%BeginExpansion
{\textstyle\bigcap\limits_{i=1}^{n}}
%EndExpansion
ann(x_{i})=%
%TCIMACRO{\tbigcap \limits_{i=1}^{n}}%
%BeginExpansion
{\textstyle\bigcap\limits_{i=1}^{n}}
%EndExpansion
(1-b_{i})\\
&  =((1-b_{1})(1-b_{2})\cdots(1-b_{n}))
\end{align*}
which is a principal ideal. Thus, $A\ $is a principal ideal von Neumann
regular ring.

$\left (ii\right) \Rightarrow\left(iii\right):$ It is straightforward.

$\left(iii\right)\Leftrightarrow\left(iv\right):$ Follows from \cite[Corollary 2.6]{Khalfi1}.

$\left(iv\right)\Rightarrow\left(ii\right):$ It is clear.

$(ii)\Rightarrow(i):\ $Assume that $A\ $is a principal ideal ring and von
Neumann regular ring. Then $A$ is reduced. Now, choose a faithful $R$-module
$E,\ $and $e\in E.\ $Since $A\ $is a principal ideal ring, $ann(e)=(a)\ $for
some $a\in A.\ $As\ $A\ $is a von Neumann regular ring, there exists an
idempotent $b\in A\ $such that $(a)=(b)=ann(1-b).\ $As $E\ $is a faithful
module, we get $ann(e)=(b)=ann((1-b)E)$, where $(1-b)\ $is a finitely
generated ideal. Thus, $E\ $is an annihilator multiplication module.
\end{proof}

Let $E$ be an $A$-module. An ideal $P$ of $A$ is defined as an \textit{associated prime} of $E$ if $P$ is a prime ideal and $P=ann(e)$ for some $0\neq e\in E$. The set of all associated primes of $E$ is denoted by $Ass_{A}(E)$ \cite{Lam}. In this context, $Ass(A)$ denotes $Ass_{A}(A)$. Associated prime ideals correspond to the irreducible components underlying a module or ideal. These ideals play a fundamental role in primary decomposition, establish connections between algebraic structures and the geometry of varieties, and are crucial in the study of dimension and depth theory. The following result presents the relationships between $Ass_{A}(E)$ and $Ass(A)$ in terms of annihilator multiplication modules.

Let $E$ be an $A$-module. The set of all torsion elements of $E$ is denoted by $T(E)=\{e\in E:ann(e)\neq0\}$. According to \cite{Chun2}, $E$ is said to be a \textit{torsion module} if $T(E)=E$; otherwise, $E$ is called a \textit{non-torsion module}.

\begin{theorem}\label{tassoc}
(i)\ Let $E\ $be an annihilator multiplication faithful $A$-module. Then $Ass_{A}(E)\subseteq Ass(A).$

(ii) Let $E\ $be an annihilator multiplication $A$-module which is also non-torsion. Then $Ass(A)=Ass_{A}(E).\ $
\end{theorem}

\begin{proof}
$(i):\ $Let $E\ $be an annihilator multiplication faithful module. Choose $P\in Ass_{A}(E).\ $Then $P\ $is a prime ideal and there exists $0\neq e\in E$ such that $P=ann(e).\ $Since $E\ $is an annihilator multiplication
module, there exists a finitely generated ideal $I\ $of $A\ $such that $ann(e)=ann(IE).\ $Then we can write $I=%
%TCIMACRO{\tsum \limits_{i=1}^{n}}%
%BeginExpansion
{\textstyle\sum\limits_{i=1}^{n}}
%EndExpansion
Aa_{i}$ for some $a_{1},a_{2},\ldots,a_{n}\in I.\ $This implies that
\begin{align*}
P  &  =ann(e)=ann(IE)\\
&  =%
%TCIMACRO{\tbigcap \limits_{i=1}^{n}}%
%BeginExpansion
{\textstyle\bigcap\limits_{i=1}^{n}}
%EndExpansion
ann(a_{i}E).
\end{align*}
Since $%
%TCIMACRO{\tbigcap \limits_{i=1}^{n}}%
%BeginExpansion
{\textstyle\bigcap\limits_{i=1}^{n}}
%EndExpansion
ann(a_{i}E)\subseteq P$ and $P\ $is a prime ideal, there exists $t\in
\{1,2,\ldots,n\}$ such that $ann(a_{t}E)\subseteq P=%
%TCIMACRO{\tbigcap \limits_{i=1}^{n}}%
%BeginExpansion
{\textstyle\bigcap\limits_{i=1}^{n}}
%EndExpansion
ann(a_{i}E)\subseteq ann(a_{t}E),\ $that is, $P=ann(a_{t}E).\ $As $E\ $is a
faithful module, we have $P=ann(a_{t}E)=ann(a_{t}).\ $This gives $P\in
Ass(A),\ $that is, $Ass_{A}(E)\subseteq Ass(A).$

$(ii):\ $Since every non-torsion module is faithful, by $(i),\ $we have
$Ass_{A}(E)\subseteq Ass(A).$\ For the reverse inclusion, choose $P\in
Ass(A).\ $Then there exists $0\neq a\in A$ such that $P=ann(a)\ $is a prime
ideal of $A.\ $Since $E\ $is a non torsion module, there exists $0\neq e\in
E\ $such that $ann(e)=(0)$ which implies that $P=ann(a)=ann(ae),$ that is, $P\in Ass_{A}(E).\ $Then we have the equiality $Ass(A)=Ass_{A}(E).\ $

\end{proof}

The following example demonstrates why the conditions of "annihilator multiplication" or "faithful" are necessary for Theorem \ref{tassoc} (i) to hold. Also, it shows that the conditions "annihilator multiplication module" or "non-torsion module" in Theorem \ref{tassoc} (ii) can not be removed.

\begin{example}\label{exassociated}
\begin{itemize}
\item [(i)] Consider the $\mathbb{Z}$-module $E=\mathbb{Z}_p\oplus\mathbb{Z}$, where $p$ is a prime number. Then $E$ is not an annihilator multiplication module by Example \ref{ex2}. Also it is easy to check that $E$ is a faithful (even, a non-torsion) $\mathbb{Z}$-module and $Ass_{\mathbb{Z}}(E)=\{(0),p\mathbb{Z}\}\nsubseteq Ass(\mathbb{Z})=\{(0)\}$.
\item [(ii)] Consider the $\mathbb{Z}$-module $E=\mathbb{Z}_p$, where $p$ is a prime number. Then $E$ is a multiplication module, so by Example \ref{prop1}, it is an annihilator multiplication module. Also, $E$ is not a faithful module since $ann(E)=p\mathbb{Z}\neq 0$. However, $Ass_{\mathbb{Z}}(E)=\{p\mathbb{Z}\}$ and $Ass(\mathbb{Z})=\{(0)\}$ are not comparable.
\item [(iii)] Consider the $\mathbb{Z}$-module $E=\bigoplus_{i=1}^{\infty} \mathbb{Z}_p$ for some prime number $p$. Then by Proposition \ref{pdirectsum} and Example \ref{prop1}, $E$ is an annihilator multiplication module. However, $E$ is a torsion module. Also, one can easily verify that $Ass(\mathbb{Z})=\{(0)\}$ and $Ass_{\mathbb{Z}}(E)=\{p\mathbb{Z}\}$ are distinct.
\end{itemize}
\end{example}

While the examples provided in Example \ref{exassociated} highlight the significance of the annihilator multiplication property in the study of associated prime ideals, the necessity of the "non-torsion" condition in Theorem \ref{tassoc} (ii) remains an intriguing question. Specifically, if $E$ is assumed to be a faithful torsion annihilator multiplication module, it remains unknown whether the equality $Ass_{A}(E) = Ass(A)$ holds without any additional assumptions, as we have yet to find a counterexample. Consequently, we leave the validity of this assertion as an open problem for the reader. 

\begin{question}
Let $E$ be an annihilator multiplication faithful torsion $A$-module. Is it true that $Ass(A)=Ass_{A}(E)$ ?
\end{question}

\vspace{0.5cm}
\noindent\textbf{Declarations}
\newline
\noindent\textbf{Conflict of Interest:} The authors declare that they have no conflict of interest.\newline
\noindent\textbf{Data Availability Statement:} Data sharing is not applicable to this article as no datasets were generated or analyzed during the current study.

\end{document}